\documentclass[10pt, notitlepage]{article}
\usepackage{amssymb,amsmath,comment}
\usepackage{epsfig,enumerate,amsmath,amsfonts,amssymb}
\usepackage{indentfirst}
\usepackage{pstricks}
\usepackage{setspace}
\usepackage{latexsym}
\usepackage{bm}
\usepackage{amsthm}
\usepackage{amsmath}
\usepackage{parskip}
\usepackage{fixmath}
\catcode`\@=11 \@addtoreset{equation}{section}
\def\thesection{\arabic{section}}

\def\theequation{\thesection.\arabic{equation}}
\catcode`\@=12
\usepackage{colortbl}
\usepackage{a4wide}

\newcommand{\ds} {\displaystyle}
\newcommand{\e}{\varepsilon}
\newcommand{\pa} {\partial}
\newcommand{\al} {\alpha}
\newcommand{\ba} {\beta}
\newcommand{\de} {\delta}
\newcommand{\ga} {\gamma}
\newcommand{\Om} {\Omega}
\newcommand{\ra} {\rightarrow}
\newcommand{\rp} {\rightharpoonup}
\newcommand{\De} {\Delta}
\newcommand{\la} {\lambda}

\newcommand{\noi} {\noindent}

\newcommand{\mb} {\mathbb}
\newcommand{\mc} {\mathcal}

\newcommand{\ld} {\langle}
\newcommand{\rd} {\rangle}

\usepackage[all]{xy}
\catcode`\@=11
\def\theequation{\@arabic{\c@section}.\@arabic{\c@equation}}
\catcode`\@=12

\def\R{{I\!\!R}}

\def\N{{I\!\!N}}

\def\QED{\hfill {$\square$}\goodbreak \medskip}

\def\R{\mathbb{R}^N}

\newtheorem{Theorem}{Theorem}[section]
\newtheorem{Lemma}[Theorem]{Lemma}
\newtheorem{Proposition}[Theorem]{Proposition}

\newtheorem{Remark}[Theorem]{Remark}
\newtheorem{Definition}[Theorem]{Definition}

\begin{document}
	
	\title
	{  Regularity  results on a class of doubly nonlocal  problems }

	\author{  Jacques Giacomoni$^{\,1}$ \footnote{e-mail: {\tt jacques.giacomoni@univ-pau.fr}}, \ Divya Goel$^{\,2}$\footnote{e-mail: {\tt divyagoel2511@gmail.com}},  \
		and \  K. Sreenadh$^{\,2}$\footnote{
			e-mail: {\tt sreenadh@maths.iitd.ac.in}} \\
		\\ $^1\,${\small Universit\'e  de Pau et des Pays de l'Adour, LMAP (UMR E2S-UPPA CNRS 5142) }\\ {\small Bat. IPRA, Avenue de l'Universit\'e F-64013 Pau, France}\\  
	$^2\,${\small Department of Mathematics, Indian Institute of Technology Delhi,}\\
	{\small	Hauz Khaz, New Delhi-110016, India } }

	\date{}
	
	\maketitle

		\begin{abstract}
		\noi The purpose of this article is twofold. First, an issue of regularity of weak solution to the problem $(P)$ (See below) is addressed. Secondly, we investigate the question of $H^s$ versus $C^0$- weighted minimizers of the functional associated to problem $(P)$ and then give applications to existence and multiplicity results. 
%		In this article, we study the regularity of the weak solution to nonlocal elliptic \textcolor{red}{equations} involving  \textcolor{red}{ a Choquard type  nonlinearity}. 
		
		\medskip
		
		\noi \textbf{Key words}: Choquard equation, fractional Laplacian, regularity, uniform bound, singular nonlinearity, local minimizers.  
		
		\medskip
		
		\noi \textit{2010 Mathematics Subject Classification:   35B45, 35R09, 35B65, 35B33 }

	\end{abstract}
\section{Introduction}
In this article we will study the following problem: 
\begin{equation*}
(P)\;
\left\{\begin{array}{rllll}
(-\De)^s u
&=g(x,u)+ \left(\ds \int_{\Om}\frac{F(u)(y)}{|x-y|^{\mu}}dy\right) f(u)  \; \text{in}\;
\Om,\\
u&=0 \; \text{ in } \R \setminus \Om
\end{array}
\right.
\end{equation*}
where $\Om$ is a smooth bounded domain in $\mathbb R^n, N\ge 2$, $s \in (0,1)$,  $\mu <N$, $g:\Om\times \mathbb R\to \mathbb R$ Carath\'edory function, $f:\mathbb R \rightarrow \mathbb R$  is a continuous  function and $F$ is the primitive of $f$.  Here the operator  $(-\De)^s$ is the fractional Laplacian   defined  up to a positive multiplicative constant as 
\begin{align*}
(-\De)^su(x)=\text{P.V. } \int_{ \R} \frac{u(x)-u(y)}{|x-y|^{N+2s}} dy
\end{align*}
where P.V. denotes the Cauchy principal value. 

The existence and  regularity  of weak solutions have been a fascinating topic  for the researchers for a long time. The work on Choquard equations was started with the  quantum theory of a polaron model given by S. Pekar \cite{pekar}.  In 1976, in the modeling of a one component plasma,  P. Choquard \cite{ehleib} used the following equation with $\mu=1,\; p=2$ and $N=3$:
\begin{equation}\label{ch26}
-\De u +u = \left(\frac{1}{|x|^{\mu}}* F(u)\right) f(u) \text{ in } \mathbb{R}^3
\end{equation} 
where $f(u)=|u|^{p-2}u$ and $F^\prime =f$. In \cite{moroz2}, Moroz and Schaftingen   established the existence of a ground state solution  and the regularity of weak solutions of  the  problem \eqref{ch26} in higher dimensions  $N \geq 3 , \mu \in (0,N)$ and with more general functions $F\in C^1(\mathbb{R},\mathbb{R})$ satisfying certain growth conditions.  %F^\prime= f$. Moreover, $f$ satisfies the following properties:\\
%$(f_1)$ there exists $C>0$  such that $ |sf(s)| \leq %C(|s|^{\frac{2N-\mu}{N}}+ |s|^{\frac{2N-\mu}{N-2}})$. \\
%$(f_2)$  $\ds \lim_{s \ra 0} \frac{F(s)}{|s|^{\frac{2N-\mu}{N}}}=0, \;  \ds\lim_{s \ra \infty} \frac{F(s)}{|s|^{\frac{2N-\mu}{N-2}}}=0$ and there exists $s_0  \in \mathbb{R}\setminus \{ 0\}$ such that $F(s_0)\not =0$. \\
For more results on the existence of solutions we refer to \cite{moroz4,moroz5} and the references therein.  In \cite{yang}, Yang and Gao  studied the  Brezis-Nirenberg type result for the  following equation 
\begin{equation*}
-\Delta u = \la u+ \left(\int_{\Om}\frac{|u(y)|^{2^*_{\mu}}}{|x-y|^{\mu}}dy\right)|u|^{2^*_{\mu}-2}u \text{ in } \Om, \quad
u=0 \text{ on } \pa \Om, 
\end{equation*}
\noi where  $\Om\subset \mb R^N, N\ge 3$ is a bounded domain having smooth boundary $\pa \Om$,  $\la>0$, $0<\mu<N$ and $ 2^*_\mu = \frac{2N-\mu}{N-2}$. Later, many researchers studied the Choquard equation for the  existence and  multiplicity  of solutions, for instance  see \cite{alves,yang1, ts} and references therein. \\

%\textcolor{blue}{[give more details about methods in particular for a priori bounds and limits of the results]}\textcolor{black}{[We have already results regarding a  priori bounds and limits of the results in page 3 in the paragraph starting as  "the main purpose of this article..."]} 
On the other hand,  in recent years, the subject of  nonlocal elliptic equations involving fractional Laplacian 
 has gained more popularity because of many  applications such as continuum mechanics,  game theory and phase transition phenomena. 
% have been studied extensively due to its wide range of  application in physical models. 
For an extensive survey on fractional Laplacian and its applications, one may refer to \cite{valdi, stinga} and references therein. 
 The nonlocal equations with Hartree-type nonlinearities were used to model the dynamics of pseudo-relativistic boson stars. In fractional quantum
mechanics, fractional Schr\"odinger equations play an important role, for instance see \cite{frank,zhang,ts}.  For the existence and multiplicity results on fractional Laplacian, readers can refer to \cite{servadei} and references therein. For the doubly nonlocal problem, precisely, the nonlocal elliptic equation involving fractional Laplacian and Choquard type nonlinearity, there are articles which discuss the existence and multiplicity of solutions,  we cite  \cite{ambrosio,avenia,pucci,zhang} and references therein, with  no attempt to provide a complete list. 
%For the works on existence and multiplicity to the problem in fractional Laplacian with Choquard type nonlinearity we refer \cite{ts}. 

  %There is ample amount of paper on the existence of solution of the following equation
%\begin{equation}\label{ch23}
%(-\Delta)^s u = f(x,u)\text{ in } \Om, \quad
%u=0 \text{ in } \R\setminus \Om, 
%\end{equation}
%where $f$ is a Caratheodory function.
 %We cite  \cite{servadei} for the works of existence of solutions  to problem \eqref{ch23}. 
 Regularity results   about problem involving fractional  diffusion  are also attracting a   large number of researchers.
% Since {\color{red} ... many researchers have explored  this subdivision of elliptic equations (refer a survey paper on regularity theory here. I AM NOT ABLE TO FIND A SURVEY PAPER OR A BOOK FOR THE REFERENCE PURPOSE)} 
%  The topic of regularity of weak solutions to  doubly nonlocal elliptic is more sensitive and require more skillfulness. 
 % \textcolor{blue}{[give some references of Silvestre,  Servadei, Valdinoci, Stein, Ros-Oton Serra, etc. with methods]}.
  Consider the following nonlocal  problem  
  \begin{equation}\label{ch24}
  (-\Delta)^s u =g\text{ in } \Om, \quad
  u=h \text{ in } \R\setminus \Om. 
  \end{equation} The interior regularity of    solutions to \eqref{ch24} is primarily determined by Caffarelli and Silvestre.  In \cite{caffe1}, authors developed the $C^{1+\al}$ interior regularity for viscosity solutions to
 nonlocal equations with bounded measurable coefficients. For the convex equation, authors proved  $C^{2s+\al}$ regularity in \cite{caffe2} while in \cite{caffe3},  authors  established  a perturbative theory for non translation invariant equations.  In \cite{silves}, Silvestre studied  regularity of weak solutions to  free boundary problem.  
 For the boundary regularity, Ros-Oton and Serra \cite{RS}  studied the regularity of weak solutions to \eqref{ch24} with $h=0$ and  $g \in L^\infty(\Om)$. By using  a suitable
 upper barrier and the interior regularity results for the fractional Laplacian they prove that $u\in C^s(\mathbb R^N)$ and $\|u\|_{C^s} \le c \|g\|_{L^\infty(\Om)}$ for some constant $c$. Moreover, authors  established  a fractional analog of the Krylov boundary Harnack method to further prove $u \in C_d^{0,\al}(\overline{\Om})$ for some $\al \in (0,1)$.
 In \cite{RS1}, authors proved  the  high integrability  of the weak solution by using the   regularity of Riesz potential established in  \cite{stein}. 
%  For more results on regularity of solutions to stationary problems involving fractional Laplacian, one can refer  to \cite{RS1,silves} and reference therein. 
  In \cite{adi}, authors discussed the existence and  regularity of  weak solution to  the following problem
 \begin{equation*}
 (-\De)^s u
 =u^{-q}+  f(u)  \; \text{in}\;
 \Om,\; 
 u=0 \; \text{ in } \R \setminus \Om
 \end{equation*}
 where $q>0$ and the function  $f$ is of subcritical growth. When $f$ has  critical growth then the question of existence and regularity have been answered in \cite{gts1}. 
% In \cite{avenia}, dAvenia, Siciliano and Squassina  prove the regularity of weqk solutions to the following problem 
%	 Attempts were made in \cite{avenia,shen} and references therein to answer the  question of  regularity of solutions fractional Laplacian and Choquard nonlinearity in whole space $\R$.

%  \textcolor{blue}{ Should we say that there are only few results that apply to special cases to highlight our approach?]}

Despite the ample amount of  research on  doubly nonlocal problems,  there is very little
done in respect of regularity of weak  solutions to these problems. For instance, in \cite{avenia}, authors proved   the regularity of a ground state solution of doubly nonlocal equation with subcritical growth in the sense of Hardy-Littlewood-Sobolev inequality, by generalizing the idea of \cite{moroz4} in fractional framework. In \cite{su}, authors establish the $L^\infty(\mathbb R)$ bound of the nonnegative ground state solution of doubly local problem with critical growth in the sense of Hardy-Littlewood-Sobolev inequality under the assumption that $\mu < \min\{ N, 4s\}$. 

 In \cite{yangjmaa}, Gao and Yang  studied the Dirichlet problem involving  Choquard nonlinearity with Laplacian operator. 
 Here authors aim to prove the regularity for  weak solutions. The boot-strap techniques    as it is  developed in \cite{yangjmaa} work for the subcritical growth 
% in the sense of Hardy-Littlewood-Sobolev inequality and the Sobolev inequality 
 and seems to fail in handling the critical non linearity in the sense of Hardy-Littlewood-Sobolev inequality.  For the critical case, Moroz and Schaftingen \cite{moroz2}, studied problem \eqref{ch26} and prove the $W^{2,p}_{\text{loc}}(\mathbb R^N), \; p>1$,  regularity of the weak solution for  problems in the whole space without a perturbation term $g(x,u)$. The techniques given in \cite{moroz2} cannot be straightforward carried to problem $(P)$ in a general setting.   The regularity of positive solution to the following singular problem 
 \begin{equation}\label{ch34}
 -\De u
 = u^{q-1} + \left(\ds \int_{\Om}\frac{F(u)(y)}{|x-y|^{\mu}}dy\right) f(u)  \; \text{in}\;
 \Om,
 u=0 \; \text{ in } \R \setminus \Om, ~ 0<q<1
 \end{equation}
 was also an open problem. 
 
 Motivated by the above discussion and  the stated issues, the  first part of the  present article is intended to  address the question of  $L^\infty(\Om) $ bound for weak solutions of the problem $(P)$ covering  large classes of $f$ and $g$.  Since once  $L^\infty(\Om)$ is there  then one can use the result given by Ros-Oton \cite{RS,silves} coupled with Hardy-Littlewood-Sobolev inequality, to prove the desired  regularity  results. 
  To prove the $L^\infty(\Om)$ bound,  we  develop an unified approach handling both subcritical and critical case of the   perturbation  $g$. 
In this article we also provide an answer to the regularity of weak solutions to doubly nonlocal equation involving singular nonlinearity, particularly  problem \eqref{ch34}. The existence and multiplicity  of solutions to problem \eqref{ch34},  is specially address in  \cite{jds}.  The novelty of the obtained results here is that they hold true for all $\mu<N$, contrasting to previous regularity results in literature.  
 The techniques and tools which are used here to prove the $L^\infty(\Om)$ estimate are contemporary and new. Precisely, we extend further the classical Brezis-Kato techniques \cite{kato} to improve the integrability of weak solutions to (P). In addition, we  mention that to the best of our knowledge, there is no  article which establish the proof of  $L^\infty(\Om)$ bound  to problem involving singular nonlinearity.  The  results in this article can be used similarly  to  Laplacian operator  (that is, $s=1$) and are also new to the literature.    
 
 The second part of this article is destined to  prove the $H^s$ versus $C^0$- weighted minimizers. That is, we show that the local minima with respect to $C_d^0(\overline{\Om})$ topology will also be a local minima with respect to $X_0$ topology. 
 In variational problems this result illustrate   a significant role as it helps to prove that the solutions to constraint minimization of the energy functional emerge as solutions to   unconstraint local minimization  of the energy functional.  
 This procedure of constraint minimizations has ample amount of applications such as to prove the  existence and multiplicity of solutions to elliptic problems, for instance see Theorem \ref{thmch4}.

 In case of local framework this result was first  done by Brezis and Nirenberg \cite{niren}.
  Here authors prove that local minima in $C^1 $ will remain so in $H^1$ topology despite of the fact that latter one is weaker than the former one.  In fractional framework, this result is proved by Iannizzotto, Mosconi and Squassina \cite{squassina}. 
  But in case of nonlocal nonlinearity, in particular, Choquard equation, a particular  case to our result had been answered by \cite{yangjmaa} for the Laplacian operator.  For the  general nonlinearity, this issue is recently posed as an open problem in \cite{ts}. 
  In this article, we also provide a  full answer to this  open problem. Since there is significant amount of difference in handling doubly nonlocal problem,  so  we cannot stick around the tools given  in \cite{niren,squassina} to  establish the  result. 
  
%In this article we prove  the $L^\infty(\Om)$ of the problem $(P)$ with  perturbation $G$.\textcolor{blue}{[again a repetition!]}
%With this introduction, 
\begin{Remark}
	We would like to remark that  the results  of  our article can be adapted to  the following fractional Schr\"odinger problem
	\begin{equation*}
	(-\De)^s u + Vu 
	=g(x,u)+ \left(\ds \int_{\Om}\frac{F(u)(y)}{|x-y|^{\mu}}dy\right) f(u)  \; \text{in}\;
	\Om,	u=0 \; \text{ in } \R \setminus \Om,
	\end{equation*}
	where $ V\in L^2(\Om)$ and $	(-\De)^s + V$ should be coercive in the energy space $X_0$. 
\end{Remark}

\section{Functional framework and main results}
This section of the article is intended to provide the fractional Sobolev space setting. For the complete and rigid details, one can refer \cite{nezzaH,servadei}. Further in this section we  state the main results of  current article with a short sketch  of proof. \\
For $0<s<1$, the fractional Sobolev space is defined as 
\begin{align*}
H^{s}(\R)= \left\lbrace u \in L^2(\R): \int_{\R}\int_{\R} \frac{|u(x)-u(y)|^2}{|x-y|^{N+2s}}~dxdy <+ \infty \right\rbrace
\end{align*}
\noi endowed with the norm 
\begin{align*}
\|u\|_{H^{s}(\R)}:= \|u\|_{L^2(\R)}+ [u]_{H^s(\mathbb R^N)} = \|u\|_{L^2(\R)}+\left(\int_{\R}\int_{\R} \frac{|u(x)-u(y)|^2}{|x-y|^{N+2s}}~dxdy  \right)^{\frac{1}{2}}.
\end{align*}
\noi Consider the space 
\begin{align*}
X_0:=  \{u \in H^{s}(\R): u=0  \text{ a.e in  } \R \setminus \Om \}
\end{align*}
  equipped with the norm 
\begin{align*}
\ld u,v \rd = \int_{Q} \frac{(u(x)-u(y))(v(x)-v(y))}{|x-y|^{N+2s}}~dxdy 
\end{align*} 
where  $ Q= \R \setminus (\Om^c\times \Om^c)$. 
From the embedding results (\cite{servadei}), the space $X_0$ is continuously embedded  into $L^r(\R)$ with $ r\in [1,2^*_s]$ where $2^*_s= \frac{2N}{N-2s}$. 
The best constant $S_s$ is defined  
\begin{align}\label{ch27}
S_s= \inf_{u \in X_0\setminus \{ 0\}}   \frac{\int_{\R}\int_{\R} \frac{|u(x)-u(y)|^2}{|x-y|^{N+2s}}~dxdy }{\left( \int_{ \Om} |u|^{2^*_s}~dx\right)^{2/2^*_s}}. 
\end{align}
Let $d: \overline{\Om} \ra \mathbb{R}_+ $ by $d(x):= \text{dist}(x,\R\setminus \Om),\; x \in \overline{\Om}$. The best constant $S_H$ is defined as 
\begin{align}\label{ch35}
S_H= \inf_{u \in X_0\setminus \{ 0\}}   \frac{\int_{\R}\int_{\R} \frac{|u(x)-u(y)|^2}{|x-y|^{N+2s}}~dxdy }{ \int_{ \Om}\frac{ |u|^{2}}{d^{2s}}~dx}. 
\end{align}
Now we define the weighted H\"older-type spaces
\begin{align*}
& C^0_d(\overline{\Om}) := \bigg \{ u \in C^0(\overline{\Om}): u/d^s \text{ admits a continuous extension to } \overline{\Om}   \bigg\},\\
& C^{0,\al}_d(\overline{\Om}) := \bigg \{ u \in C^0(\overline{\Om}): u/d^s \text{ admits a } \al \text{ -H\"older continuous extension to } \overline{\Om}   \bigg\}
\end{align*}
endowed with the norms
\begin{align*}
\|u\|_{0,d}:= \|u/d^s\|_{\infty}, \quad \|u\|_{\al,d}:= \|u\|_{0,d}+ \sup_{x, y \in \overline{\Om}, x\not = y } \frac{|u(x)/(x)d^s- u(y)/d(y)^s|}{|x-y|^{\al}}
\end{align*}
respectively.  
We assume that $f$ satisfies the following growth conditions throughout the current article.\\
$(\mc F)  \quad F\in C^1(\mathbb{R},\mathbb{R})$, $  F^\prime= f$ and  there exists $C>0$  such that for all $t\in \mathbb R$, 
\begin{align*}
|tf(t)| \leq C(|t|^{\frac{2N-\mu}{N}}+ |t|^{\frac{2N-\mu}{N-2s}}).
\end{align*}
\begin{Definition} A function $u\in X_0$ with $u\equiv 0$ in $\mathbb R^N \backslash \Omega$ is said to be a solution  to (P) if 
	\begin{equation*}
	\int_{Q} \frac{(u(x)-u(y))(\phi(x)-\phi(y))}{|x-y|^2}~dxdy= \la\int_{\Om} g(x,u )\phi ~dx + \iint_{\Om\times \Om} \frac{F(u)f(u)}{|x-y|^\mu} \phi ~dxdy  
	\end{equation*}
	 for all $\phi \in X_0$. 
	\end{Definition}
Let $G(x,u) = \int_{0}^{u} g(x, \tau )~d\tau$ then  functional associated with problem $(P)$ is defined as 
\begin{align*}
J(u) = \frac{\|u\|^2}{2}-\int_{\Om}G(x,u)~dx -\frac{1}{2} \iint_{\Om\times \Om} \frac{F(u)F(u)}{|x-y|^\mu} ~dxdy, \text{ for all } u \in X_0. 
\end{align*}
 With this  functional framework, we state the  main results of the article. First we state the result about the regularity of weak solution to problem $(P)$. 
%  First we consider the problem motivated from the convex-concave growth problems.
%For the superlinear growth nonlinearity problems, we have 
\begin{Theorem}\label{thmch1}
		Let  $ g: \overline{\Om} \times \mathbb{R} \ra \mathbb{R}$ be  a Carath\'eodory function  satisfying
	\begin{align*}
	g(x, u) =		O(|u|^{2^*_s-1}), & \text{ if } |u|\ra \infty
	\end{align*}
	uniformly for all $ x \in \overline{\Om}$.   Then any solution $u \in X_0$ of  $(P)$ belongs to $L^\infty(\R) \cap C^s(\R)$. 	Furthermore, there exists positive constant $C$ depending on $N,\mu, s, |\Om|$  such that\\ $|u|_\infty\leq C (1+|u|_{2^*_s})^{\frac{2}{(2^*-1)(2^*-2)}}\left( 1+ \left( (1+|u|_{2^*_s}) \left(|u|_{2^*_s}^{2^*_s} + R^{2^*_s} |u|_{2^*_s-1}^{2^*_s-1} \right) \right)^{\frac{2^*_s}{2}} 
	\right)^{\frac{2}{2^*_s(2^*_s-1)}}$
	and $R>0$ large enough such that $	\left( \int_{|u|>R} |u|^{2^*_s}  ~dx\right)^{\frac{2^*_s-2}{2^*_s}} \leq \frac{1}{2 C (1+|u|_{2^*_s})}.$
%	Furthermore, there exists $S\in C(\mathbb{R}^2)$ such that $|u|_{\infty}\leq S(K_0, |u|_{2^*_s})$ where $K_0$ is a 
\end{Theorem}
Next we consider the regularity for singular problems.
\begin{Theorem}\label{thmch2}
%Let 	$ q \in (0,1)$ and $ g(x,u)=u^{q-1}$.   Then any positive solution $u \in X_0$ of  $(P)$ belongs to $L^\infty(\R) \cap C^s(\R)$. Moreover,  there exists a continuous function such that $M\in C((\mathbb{R}^+)^2)$ such that $|u|_\infty\leq M(R,|u|_{2^*_s}) $ where  $R$  is given by \eqref{Restimate}. 
Let 	$ q \in (0,1)$ and $ g(x,u)=u^{q-1}$.   Then any positive solution $u \in X_0$ of  $(P)$ belongs to $L^\infty(\R) \cap C^s(\R)$. Moreover,  there exists $C>0 $ depending on $N,\mu, s $ and  $|\Om|$ and a positive constant $C_1$  s.t.
\begin{align*}
|u|_{\infty} \leq 1+  C_1\mc S_1^{\frac{2}{(2^*-1)(2^*-2)}}\left( 1+ \left( \mc S_1\left(|(u-1)^+|_{2^*_s}^{2^*_s} + R^{2^*_s} |(u-1)^+|_{2^*_s-1}^{2^*_s-1} \right) \right)^{\frac{2^*_s}{2}} 
\right)^{\frac{2}{2^*_s(2^*_s-1)}}  
\end{align*}
 with $\mc S_1= \max\{ 1, C(N,\mu ,|\Om|) |u|_{2^*_s} \}$, $R>0$ such that  $\left( \int_{|u|>R} |(u-1)^+|^{2^*_s}  ~dx\right)^{\frac{2^*_s-2}{2^*_s}} \leq \frac{1}{2(2^*_s+1)\mc S_1}.$
\end{Theorem}
\begin{Remark}
Replacing $u^{q-1}$ by $g(x,u)$ with $g:\Om\times \mathbb R^+\backslash\{0\}\to \mathbb R^+$ satisfying $g(x,t)t^{1-q}$ uniformly bounded as $t\to 0^+$ and $t\to g(x,t)$ nonincreasing for a.e $x\in \Om$, then Theorem \ref{thmch2} holds.
\end{Remark}
To achieve the intended goal in the above results, we first prove the non local version of Brezis-Kato estimates (See Lemma \ref{lemch3} and \ref{lemch4}) in a similar manner as in \cite{kato, moroz2}. Subsequently we construct a sequence of  coercive, bilinear maps. This sequence  allows us to  further construct a sequence of function $u_n$ will converge weakly to $u$ (weak solution to $(P)$). Then  we inherit some  classical technique of Brezis-Kato \cite{kato, moroz2}. We prove that $u_n \in L^p(\Om) $ with $2^*_s<p< p_0$ for some $p_0$.  Consequently, $u \in L^p(\Om) $ with $2^*_s<p< p_0$. Using these estimates,  we establish
\begin{align*}
\int_{\Om}\frac{F(u(y))}{|x-y|^{\mu}}dy \in L^\infty(\Om). 
\end{align*}
 Then by Moser iterations proved  established in Lemma  \ref{lemch5},  we prove that $u \in L^\infty(\Om) $. For the $C^{0,\al}(\overline{\Om})$ regularity we can conclude by using   Ros-Oton and Serra \cite{RS} mentioned above. We mention here that the  construction of  the bilinear forms for the Theorem \ref{thmch2} is most sensitive part and require more technicality. We remark that if we use Moser iterations without employing the method we present  above then we can achieve   $ L^\infty(\Om) $ bound of weak solutions to $(P)$ under the  additional assumption $\mu < \min \{N, 4s  \}$ and $f= |u|^{\frac{N-\mu+2}{N-2s}}$, see for instance \cite{jds}. To incorporate  the case $\mu \geq  \min \{N, 4s  \}$, we develop the    above stated  unified course of steps. 
%  For the $C^s$ regularity we used the result given by Ros-oton and Serra \cite{RS} mentioned above.

The second main aim of this paper is to give an application of $L^\infty(\Om)$ estimate. In that direction we have the following.
\begin{Theorem}\label{thmch3}
		Let  $ g: \overline{\Om} \times \mathbb{R} \ra \mathbb{R}$ be  a Carath\'eodory function  satisfying
	\begin{align*}
	g(x, u) =		O(|u|^{2^*_s-1}), & \text{ if } |u|\ra \infty
	\end{align*}
	uniformly for all $ x \in \overline{\Om}$. Let $v_0 \in X_0$. Then the following assertions holds are equivalent:
	\begin{itemize}
			\item[(i)] there exists $\e>0 $ such that $ J(v_0+v) \geq J(v_0) $ for all $  c\in X_0,\;  \|v\|\leq \e$. 
		\item[(ii)] there exists $\rho>0$ such that $J(v_0+ v) \geq J(v_0)$ for all $ v \in X_0 \cap C^0_d(\overline{\Om})$, $\|v\|_{0,\de}\leq \rho$. 
		\end{itemize}
\end{Theorem}
To prove the  above result we have modified the techniques which have been developed by  \cite{niren,squassina}.

As an application of the  $H^s$ versus $C^0$- weighted minimizers, in section 6, we proved the existence of weak solution to Choquard equation, which is also a local minimizer in $X_0$ topology (See Theorem \ref{thmch4}). To prove the desired result, instead by trapping the nonlinearity between sub and supersolution, we  generalize Perron's method for the doubly nonlocal problem \cite[Theorem 2.4]{struwe2}. An advantage to proceed by this alternative method is that we  don't need  strong assumptions on sub and supersolution except the fact, they belong to $X_0$. 
%\begin{Remark}
% We remark that one can prove $H^s$ versus $C^0_d$ (Theorem \ref{thmch3}) in singular case as well. But in this %case we need  a   strong  comparision  principle for an associated  Schr\"odinger problem with unbounded %nonnegative potential  to prove that Theorem \ref{thmch3}(ii) holds.  We point out here that the  result proved in   %\cite{servadei1} applies to bounded potentials case.    
%\end{Remark}

For simplicity of illustration, we set some notations. We denote $\|u\|_{L^p(\Om)} $ by  $|u|_p$ and  $\|u\|_{X_0}$ by $\|u\|$. $B^X_\rho(u),\bar{B}^X_\rho(u)\;(B^d_\rho(u),\bar{B}^d_\rho(u))$ denote the open and closed ball, centered at $u$ with radius $\rho$, respectively in $X_0\;(C^0_d(\overline{\Om}))$. The positive constant $C$ values change case by case.

Rest of the paper organized as follows: In section 3, we give some preliminary results. In section 4, we give some technical lemmas which will help us to prove the main theorems of the paper. In section 5, we prove the Theorem \ref{thmch1} and \ref{thmch2}. In section 6, we give the proof of Theorem \ref{thmch3} and provide an application to Theorem \ref{thmch3}.

\section{Preliminary results}
In this section we contribute  some preliminary results, though rather straightforward, do  not  appear explicitly in former literature,  and are worthy to archive  them here.  \\
The Hardy-Littlewood-Sobolev Inequality, foundational in study of Choquard equation is stated here.   
\begin{Proposition} \cite{leib}
	Let $t,r>1$ and $0<\mu <N$ with $1/t+\mu/N+1/r=2$, $f\in L^t(\mathbb{R}^N)$ and $h\in L^r(\mathbb{R}^N)$. There exists a sharp constant $C(t,r,\mu,N)$ independent of $f,h$, such that 
	\begin{equation*}\label{co9}
	\int_{\mathbb{R}^N}\int_{\mathbb{R}^N}\frac{f(x)h(y)}{|x-y|^{\mu}}~dydx \leq C(t,r,\mu,N) |f|_{t}|h|_{r}.
	\end{equation*}
%	If $t=r=2N/(2N-\mu)$, then 
%	\begin{align*}
%	C(t,r,\mu,N)=C(N,\mu)= \pi^{\frac{\mu}{2}}\frac{\Gamma(\frac{N}{2}-\frac{\mu}{2})}{\Gamma(N-\frac{\mu}{2})}\left\lbrace \frac{\Gamma(\frac{N}{2})}{\Gamma(\frac{\mu}{2})}\right\rbrace^{-1+\frac{\mu}{N}}.
%	\end{align*}
%	Equality holds in  \eqref{co9} if and only if $f\equiv (constant)h$ and 
%	\begin{align*}
%	h(x)= A(\gamma^2+|x-a|^2)^{(2N-\mu)/2},
%	\end{align*} 
%	for some $A\in \mathbb{C}, 0\neq \gamma \in \mathbb{R}$ and $a \in \mathbb{R}^N$. \QED
\end{Proposition}
%We define  
%\begin{align*}
%S= \inf_{u \in X_0\setminus \{ 0\}}   \frac{\int_{\R}\int_{\R} \frac{|u(x)-u(y)|^2}{|x-y|^{N+2s}}~dxdy }{\left( \int_{ \Om} |u|^{2^*_s}\right)^{1/2^*_s_\mu}}
%\end{align*}

\begin{Lemma}\label{lemch3}
	If $V\in L^\infty(\Om)+ L^{N/2s}(\Om)$ then for every $\e>0$ there exists $C_\e$ such that  for every $u \in X_0$, we have 
	\begin{align*}
	\int_{ \Om} V|u|^2~dx \leq \e^2 \|u\|^2+ C_\e \int_{ \Om}|u|^2~dx.
	\end{align*}
\end{Lemma}
\begin{proof}
Let $V=V_1+V_2$ where $V_1 \in L^\infty(\Om)$ and $V_2 \in L^{N/2s}(\Om)$. For each $k>0$ we have 
\begin{align*}
	\int_{ \Om} V|u|^2~dx & \leq  \|V_1\|_{L^\infty(\Om)}\int_{ \Om}|u|^2~dx + k \int_{|V_2|\leq k} |u|^2~dx + \int_{|V_2|> k} |V_2||u|^2~dx\\
	&\leq  \|V_1\|_{L^\infty(\Om)}\int_{ \Om}|u|^2~dx + k \int_{|V_2|\leq k} |u|^2~dx + S^{-1}_s  \left(\int_{|V_2|> k} |V_2|^{N/2s}~dx\right)^{2s/N}\|u\|^2
\end{align*}
where $S_s$ is the best constant of the embedding $X_0$ into $L^{\frac{2N}{N-2s}}.$
For a given $\e>0$, choose $k>0$ such that 
\begin{align*}
S^{-1}_s  \left(\int_{|V_2|> k} |V_2|^{N/2s}~dx\right)^{2s/N} < \e^2.
\end{align*}
It implies that 
\begin{align*}
\int_{ \Om} V|u|^2~dx & \leq   \e^2\|u\|^2+ C_\e \int_{ \Om}|u|^2~dx.
\end{align*}
\QED
\end{proof}

\begin{Lemma}\label{lemch4} \cite[Lemma 3.3]{moroz2}
	Let $p,q,r,t \in [1,\infty)$ and $\la \in [0,2]$ such that 
	\begin{align*}
	1+\frac{N-\mu}{N}-\frac{1}{p} -\frac{1}{t} = \frac{\la}{q}+ \frac{2-\la}{r}.
	\end{align*}
	If $\theta \in (0,2)$ satisfies 
	\begin{align*}
& 	\min\{ q,r\} \left(\frac{N-\mu}{N}-\frac{1}{p}  \right)< \theta < \max\{ q,r\} \left(1-\frac{1}{p}  \right)\\
& \min\{ q,r\} \left(\frac{N-\mu}{N}-\frac{1}{t}  \right)< 2-\theta < \max\{ q,r\} \left(1-\frac{1}{t}  \right)
	\end{align*}
	then for $H \in L^p(\R), K \in L^t(\R)$ and $ u \in L^q(\R) \cap L^r(\R)$,
	\begin{align*}
	\int_{ \R}(|x|^{-\mu}* (H|u|^{\theta}))K|u|^{2-\theta}~ dx \leq  C \|H\|_{L^p(\mathbb R^N)}  \|K\|_{L^t(\mathbb R^N)} 
	\left( \int_{ \R} |u|^q\right)^{\la/q} \left( \int_{ \R} |u|^r\right)^{\frac{(2-\la)}{r}}  .
	\end{align*}
\end{Lemma}
\begin{Lemma}\label{lemch1}
	Let  $N\geq 2s ,\; 0<\mu<N$ and $\theta \in (0,2)$. If $H,\; K \in L^{\frac{2N}{N-\mu+2s}}(\R) + L^{\frac{2N}{N-\mu}}(\R) $ and $ 1-\frac{\mu}{N} < \theta <1+ \frac{\mu}{N}$ then for every $\e>0$ there exists $C_{\e,\theta} \in \mathbb{R}$ such that for every $u \in H^s(\R)$,
	\begin{align*}
	\int_{ \R}(|x|^{-\mu}* (H|u|^{\theta}))K|u|^{2-\theta}~ dx \leq \e^2 \left(\int_{ \mathbb{R}^{2N}}\frac{|u(x)-u(y)|^2}{|x-y|^{N+2s}}~dxdy\right)^2 + C_{\e,\theta} \int_{ \R}|u|^2~ dx.
	\end{align*}
\end{Lemma}
\begin{proof}
	We  follow the proof of \cite[Lemma 3.2]{moroz2} in the nonlocal  framework.
	Let $H= H_1+H_2$ and $K= K_1+K_2$ with $H_1,K_1 \in L^{\frac{2N}{N-\mu}}(\R)$ and $H_2, K_2 \in L^{\frac{2N}{N-\mu+2s}}(\R)$. Now using Lemma \ref{lemch4} iteratively with appropriate values of $p,q,r,t, \theta$ and $\la$ (See \cite[Lemma 3.2]{moroz2}),  we have 
\begin{align*}
\int_{ \R}(|x|^{-\mu}* (H|u|^{\theta}))K|u|^{2-\theta}~ dx & \leq C   \left(  |H_2|_{\frac{2N}{N-\mu+2s}} +    |K_2|_{\frac{2N}{N-\mu+2s}} \right)^2 \left(\int_{ \mathbb{R}^{2N}}\frac{|u(x)-u(y)|^2}{|x-y|^{N+2s}}~dxdy\right)^2\\& \quad + C   \left(  |H_1|_{\frac{2N}{N-\mu}} +    |K_1|_{\frac{2N}{N-\mu}} \right)^2 \int_{ \R}|u|^2~ dx. 
\end{align*}
For given $\e>0$, choose $H_2, K_2$ such  that 
\begin{align*}
|H_2|_{\frac{2N}{N-\mu+2s}} ,    |K_2|_{\frac{2N}{N-\mu+2s}} < \frac{\e}{2\sqrt{C}}. 
\end{align*}
Therefore, the result holds. \QED
\end{proof}

\begin{Lemma}\label{lemch2}
	For $a,b \in \mathbb R, r\geq 2 , k\geq 0 $, we have
	\begin{align*}
\frac{4(r-1)}{r^2} \left(  |a_k|^{r/2} -|b_k|^{r/2}\right)^2 \leq 	(a-b)(a_k|a_k|^{r-2}-b_k|b_k|^{r-2}) 
	\end{align*}
	where \begin{align*}
	a_{k} =  \max\{-k , \min\{  a,k\} \}=\left\{
	\begin{array}{ll}
	-k ,  &  \text{ if }  a\leq -k, \\
	a, & \text{ if } -k< a<k,\\
	k ,  &  \text{ if }  a\geq k. \\
	\end{array} 
	\right.
	\end{align*}
	\end{Lemma}
\begin{proof}
	From \cite[Lemma 3.1]{squassina}, we have 	
	\begin{align}\label{ch1}
	\frac{4(r-1)}{r^2} \left(  a|a_k|^{\frac{r}{2}-1} -b|b_k|^{\frac{r}{2}-1} \right)^2 \leq 	(a-b)(a_k|a_k|^{r-2}-b_k|b_k|^{r-2}). 
	\end{align}
	By symmetry of  the inequality, it is enough to show that result hold for $a\leq b$. For this, let  $a= a_k $ and $b=b_k$ in \eqref{ch1}, we have 
	\begin{align*}
	\frac{4(r-1)}{r^2} \left(  a_k|a_k|^{\frac{r}{2}-1} -b_k|b_k|^{\frac{r}{2}-1} \right)^2 \leq 	(a_k-b_k)(a_k|a_k|^{r-2}-b_k|b_k|^{r-2}). 
	\end{align*}
	\begin{enumerate}
		\item[Case 1:] $0\leq b<a$\\
	Clearly $0\leq b_k< a_k$ and $a_k-b_k \leq a-b$. This implies \[(a_k-b_k)(a_k|a_k|^{r-2}-b_k|b_k|^{r-2}) \leq (a-b)(a_k|a_k|^{r-2}-b_k|b_k|^{r-2}).\]
	\item[Case 2:] $b\leq 0 \leq a$\\
	Again notice that $b_k\leq 0\leq a_k , a_k-b_k \leq a-b $ and $ a_kb_k \leq |a_kb_k|$ we have \[(a_k-b_k)(a_k|a_k|^{r-2}-b_k|b_k|^{r-2}) \leq (a-b)(a_k|a_k|^{r-2}-b_k|b_k|^{r-2})\] and \[\left(  |a_k|^{r/2} -|b_k|^{r/2}\right)^2 \leq 
	\left(  a_k|a_k|^{\frac{r}{2}-1} -b_k|b_k|^{\frac{r}{2}-1} \right)^2.\] Hence the proof. \QED
		\end{enumerate}
\end{proof}
\section{Technical results}
This section is devoted to the study of weak solutions to the following problem
\begin{equation*}
(P_1)\;
\left\{\begin{array}{rllll}
(-\De)^s u
&=g(x,u)+ \left(\ds \int_{\Om}\frac{H(y)u(y)}{|x-y|^{\mu}}dy\right) K(x) \; \text{in}\;
\Om,\\
u&=0 \; \text{ in } \R \setminus \Om,
\end{array}
\right.
\end{equation*}
where $H, K \in L^{\frac{2N}{N-\mu+2s}}(\Om) + L^{\frac{2N}{N-\mu}}(\Om) $.  Here we use the results, established in last section to improve the integrability regularity of weak solutions to the above mentioned problem. 

\begin{Proposition}\label{Propch1}
	Let $H, K \in L^{\frac{2N}{N-\mu+2s}}(\Om) + L^{\frac{2N}{N-\mu}}(\Om) $. Let 
	 $g: \overline{\Om} \times \mathbb{R} \ra \mathbb{R}$ be  a continuous   function satisfying
	\begin{align*}
	g(x, u) =		O(|u|^{2^*_s-1}), & \text{ if } |u|\ra \infty
	\end{align*}
	uniformly for all $ x \in \overline{\Om}$. Then any solution $u \in X_0$ of the problem $(P_1)$	belongs to $L^r(\Om)$ where $r \in [2, \frac{2N^2}{(N-\mu)(N-2s)})$. 
\end{Proposition}
\begin{proof} 
For $\theta=1 $ in Lemma \ref{lemch1}, there exists $\al >0$ such that for every $ \phi \in X_0$, 
\begin{align}\label{ch2}
\int_{ \Om} \int_{\Om}\frac{|H(y)\phi(y)K(x)\phi(x)|}{|x-y|^{\mu}}~dxdy \leq \frac12 \left(\int_{ Q}\frac{|\phi(x)-\phi(y)|^2}{|x-y|^{N+2s}}~dxdy\right)^2 + \frac{\al}{2} \int_{ \Om}|\phi|^2~ dx.
\end{align}
	If $3\leq 2^*_s\leq 2^*_s$ then $|u|^{2^*_s-2} \in L^{N/2s}(\Om)$. If $2<2^*_s<3$ then choose $p>1$ such that $1\leq \frac{p(2^*_s-2)N}{2s}\leq 2^*_s$ then using H\"older's inequality gives us 
	\begin{align*}
	\left( \int_{\Om} |u|^{\frac{(2^*_s-2)N}{2s}}~dx\right)^{2s/N} \leq C \left( \int_{ \Om} |u|^{\frac{p(2^*_s-2)N}{2s}}~dx\right)^{2s/Np}<\infty.
	\end{align*}
 Choose $L_1>0$  such that $ 	\left( \int_{|u|>L_1} |u|^{\frac{(2^*_s-2)N}{2s}}~dx\right)^{2s/N} \leq \frac{S_s}{2}$ where $S_s$ is the best Sobolev constant defined in \eqref{ch27}.  
	Since $g(x,u)=  O(|u|^{2^*_s-1})$ for $u$ large enough, there exist $L/2> L_1>0$ such that $ g(x,u)\leq |u|^{2^*_s-1}$ uniformly for $x \in \overline{\Om}$ and $|u|>L/2$. 
Define $\eta \in C_c^\infty[0,\infty)$ such that $0\leq \eta\leq 1$ and 
\begin{align*}
\eta(u) =  \left\{
\begin{array}{ll}
1 ,  &  \text{ if }  |u|< L/2, \\
0, & \text{ if } |u|>L.
\end{array} 
\right.
\end{align*}
Define $V:= (1-\eta) \frac{g(x,u)}{u}$ and $T := \eta g(x,u)+   \al u $. By the choice of $\eta$,  we obtain
\begin{align}\label{ch28}
|V|_{N/2s}< S_s/2 \text{ and }  T \in X_0^\prime. 
\end{align}
Observe  that  $u$ is the unique solution to the following problem 
\begin{equation*}
(-\De)^s u + \al u 
=Vu+ \left(\ds \int_{\Om}\frac{H(y)u(y)}{|x-y|^{\mu}}dy\right) K   +T  \; \text{in}\;
\Om,u=0 \; \text{ in } \R \setminus \Om.
\end{equation*}
Choose sequence $ \{  H_n \}_{n\in \N}$ and $\{  K_n \}_{n\in \N}$ in $L^{\frac{2N}{N-\mu}}(\Om)$ such that $|H_n |\leq |H|,\; |K_n|\leq |K| $ and $H_n \ra H $, $K_n \ra K $ a.e in $\Om$. 
For each $ n\in \N$, $V_n$ denotes the truncated potential defined as $V_n= V$ if $|V|\leq n$ and $V_n = n $ if $|V|>n$.  Now we introduce the bilinear form 
\begin{align*}
B_n(v, w )= & \int_{ Q} \frac{(v(x)-v(y))(w(x)-w(y))}{|x-y|^{N+2s}}~dxdy +\al \int_{ \Om} vw~dx \\&  \quad  - \int_{ \Om} \int_{\Om}\frac{H_n(y)v(y)K_n(x)w(x)}{|x-y|^{\mu}}~dxdy- \int_{ \Om}V_n vw~dx.
\end{align*}
	In view of H\"older's inequality, Sobolev embedding, \eqref{ch28}  and  \eqref{ch2}, one can easily conclude that $ B_n$ is continuous coercive bilinear form.  Hence by Lax-Miligram Lemma (See \cite[Corollary 5.8]{brezis}) there exists a   unique  $u_n \in X_0$  such that for all $w \in X_0$ we have 
\begin{align}\label{ch3}
B_n(u_n, w)= \int_{ \Om}T w~dx .
\end{align}
 Subsequently,  $u_n$ is a unique solution to  the problem
\begin{equation}\label{ch12}
(-\De)^s u_n + \al u_n 
= \left(\ds \int_{\Om}\frac{H_n(y)u_n(y)}{|x-y|^{\mu}}dy\right) K_n +V_nu_n +T  \; \text{in}\;
\Om,	u_n=0 \; \text{ in } \R \setminus \Om.
\end{equation}
Furthermore, using \eqref{ch3} we can easily prove that $u_n $ is a bounded sequence in $X_0$. It implies that up to a subsequence,  $u_n \rp u$ weakly in $X_0$. Let     $u_{n, \tau}=  \max\{-\tau , \min\{  u_n,\tau \} \} $ for  $ \tau>0$ and $x \in \Om$.  Testing  Problem \eqref{ch12} with  $ \phi = |u_{n, \tau}|^{r-2} u_{n, \tau} \in X_0$ ($2\leq r< \frac{2N}{N-\mu}$), with the help of Lemma \ref{lemch2}, we get  
\begin{equation}\label{ch6}
\begin{aligned}
&	\frac{4(r-1)}{r^2} \| |u_{n, \tau}|^{r/2}\|^2 + \al \int_{ \Om} || u_{n, \tau}|^{r/2}|~dx \\
& \leq \int_{Q} \frac{(u_n(x)-u_n(y))(\phi(x)-\phi(y))}{|x-y|^{N+2s}}~dxdy + \al \int_{ \Om} u_n \phi~dx \\
& =  \int_{\Om} \int_{ \Om} \frac{H_n(y)u_n(y)K_n(x) |u_{n, \tau}|^{r-2} u_{n, \tau} }{|x-y|^{\mu}}dy  + \int_{ \Om} V_n u_n   |u_{n, \tau}|^{r-2} u_{n, \tau}~dx + \int_{\Om}T |u_{n, \tau}|^{r-2} u_{n, \tau}~dx. 
\end{aligned}
\end{equation}
Using Lemma \ref{lemch1} with $\e^2 =\frac{(r-1)}{r^2}$, we obtain
\begin{equation}\label{ch7}
\begin{aligned}
\int_{\Om} \int_{ \Om} \frac{H_n(y)u_n(y)K_n(x) |u_{n, \tau}|^{r-2} u_{n, \tau} }{|x-y|^{\mu}}dxdy &  \leq \int_{\Om} \int_{ \Om} \frac{|H_n(y)u_{n,\tau}(y)||K_n(x)| |u_{n, \tau}(x)|^{r-1} }{|x-y|^{\mu}}dxdy\\
&  + \int_{E_{n,\tau}} \int_{ \Om} \frac{|H_n(y)u_{n}(y)||K_n(x)| |u_{n}(x)|^{r-1}}{|x-y|^{\mu}}dxdy  \\
&\leq  \frac{(r-1)}{r^2} \||u_{n \tau}|^{r/2}\|^2 + C_r \int_{ \Om}|u_{n, \tau}|^{r}~dx \\
& + \int_{E_{n,\tau}} \int_{ \Om} \frac{|H_n(y)u_{n}(y)||K_n(x)||u_{n}(x)|^{r-1}}{|x-y|^{\mu}}dxdy
\end{aligned}
\end{equation}
where $E_{n,\tau}= \{ x \in \mathbb{R}^N : |u_n(x)|\geq \tau \}$. 	By Hardy-Littlewood-Sobolev inequality and H\"older's inequality, we have 
\begin{align}\label{ch5}
\int_{E_{n,\tau}} \int_{ \Om} \frac{|H_n(y)u_{n}(y)||K_n(x)||u_{n}(x)|^{r-1}}{|x-y|^{\mu}}dxdy\leq C \left( \int_{ \R} \bigg||K_n||u_n|^{r-1}\bigg|^j~ d \xi\right)^{\frac{1}{j}}\left( \int_{E_{n,\tau}} |H_nu_n|^l~ d \xi\right)^{\frac{1}{l}}
\end{align}
where $j$ and $l$ satisfy the relation   $\frac{1}{j}= 1+ \frac{N-\mu}{2N}-\frac{1}{r}$ and $\frac{1}{l}=  \frac{N-\mu}{2N}+\frac{1}{r}$. Using the fact that $H_n, K_n \in L^{\frac{2N}{N-\mu}}(\Om)$ and  again the  H\"older's inequality,  $u_n \in L^r(\R)$ implies that  $|K_n||u_n|^{r-1} \in L^j(\R)$ and $ |H_nu_n| \in L^l(\R)$. Therefore, as $\tau\ra \infty$, \eqref{ch5} gives
\begin{align}\label{ch8}
\lim_{\tau \ra \infty}\int_{E_{n,\tau}} \int_{ \Om} \frac{|H_n(y)u_{n}(y)||K_n(x)||u_{n}(x)|^{r-1}}{|x-y|^{\mu}}dxdy =0.
\end{align}
Using the Sobolev inequality, \eqref{ch6}, \eqref{ch7} and \eqref{ch8}, we have 
 \begin{equation}\label{ch13}
 \begin{aligned}
 \frac{3(r-1)S_s}{r^2}&  \left( \int_{ \Om} |u_{n,\tau}|^{\frac{rN}{(N-2s)}}~dx\right)^{\frac{N-2s}{N}}\\
 & \leq  C_r \int_{ \Om}|u_{n}|^{r} + \int_{E_{n,\tau}} \int_{ \Om} \frac{|H_n(y)u_{n}(y)||K_n(x)||u_{n}(x)|^{r-1}}{|x-y|^{\mu}}dxdy \\
 & \quad   + \int_{ \Om} V_n u_n   |u_{n, \tau}|^{r-2} u_{n, \tau}~dx + \int_{\Om}g |u_{n, \tau}|^{r-2} u_{n, \tau}~dx. 
 \end{aligned}
 \end{equation}
Employing the fact that $g$ is a Carath\'eodory function,
 \begin{equation}\label{ch14}
 \begin{aligned}
  \int_{\Om}T |u_{n, \tau}|^{r-2} u_{n, \tau}~dx & \leq  \int_{|u|\leq L }g(x,u) |u_{n}|^{r-1} ~dx + \al \int_{\Om} u  |u_{n, \tau}|^{r-1} ~dx \\
  & \leq C(L_1) \left(  \int_{\Om}  |u|^{r} ~dx  + \int_{\Om}   |u_{n}|^{r} ~dx\right).
 \end{aligned}
 \end{equation}
 By Lemma \ref{lemch3} for $\e^2 = \frac{r-1}{r^2}$, we have 
 \begin{equation}\label{ch15}
 \begin{aligned}
 \int_{ \Om} V_n u_n   |u_{n, \tau}|^{r-2} u_{n, \tau}~dx & \leq 2 \int_{E_{n,\tau}} V_n   |u_{n}|^{r} ~dx + \int_{ \Om} V_n    |u_{n, \tau}|^{r} ~dx \\
 & \leq \frac{(r-1)}{r^2} \||u_{n \tau}|^{r/2}\|^2 + C_r \int_{ \Om}|u_{n, \tau}|^{r}~dx + 2 \int_{E_{n,\tau}} V_n   |u_{n}|^{r} ~dx. 
 \end{aligned}
 \end{equation}
 Using Dominated Convergence theorem, one can easily shows that $\ds \lim_{\tau \ra \infty} \int_{E_{n,\tau}} V_n   |u_{n}|^{r} ~dx=0$. Now taking into account \eqref{ch13}, \eqref{ch14}, \eqref{ch15} and  letting $\tau \ra \infty$, we have 
 \begin{equation*}
 \begin{aligned}
  \left( \int_{ \Om} |u_{n}|^{\frac{rN}{(N-2s)}}~dx\right)^{\frac{N-2s}{N}}
 & \leq C_r \left( \int_{\Om}|u_{n}|^r~dx + \int_{\Om}|u|^r~dx \right). 
 \end{aligned}
 \end{equation*}
 Therefore, 
 	\begin{align*}
 \limsup_{n \ra \infty} \left( \int_{ \Om} |u_{n}|^{\frac{rN}{(N-2s)}}~dx\right)^{\frac{N-2s}{N}} \leq C_r \limsup_{n \ra \infty}  \left( \int_{\Om}|u_{n}|^r~dx + \int_{\Om}|u|^r~dx \right). 
 \end{align*}
Hence, by iterating a  finite number of times, we infer that  $ u \in L^q(\Om)$ for all  $ q\in \left[2, \frac{2N^2}{(N-\mu)(N-2s)}\right) $. Moreover, there exists a  positive constant $C(q,N,\mu, |\Om|)$ such that $|u|_q\leq C(q,N,\mu, |\Om|) |u|_{2^*_s}$. 
 \QED 
\end{proof}
\begin{Definition} 
	For $\phi \in C^0(\overline{\Om})$ with $\phi >0$ in $\Om$, the set $C_\phi(\Om)$ is defined as 
	\begin{align*}
	C_\phi(\Om)= \{ u \in C^0(\overline{\Om})\; :\; \text{there exists } c \geq 0 \text{ such that } |u(x)|\leq c\phi(x), \text{ for all } x \in \Om   \},
	\end{align*}
	endowed with the natural norm $\bigg\|\ds \frac{u}{\phi}\bigg\| _{L^{\infty}(\Om)}$.
\end{Definition}
\begin{Definition}
	The positive cone of $C_\phi(\Om)$ is the open convex subset of $C_\phi(\Om)$ defined as 
	\begin{align*}
	C_\phi^+(\Om)= \left\{ u \in C_\phi(\Om)\; :\; \inf_{x \in \Om} \frac{u(x)}{\phi(x)}>0 \right\}.
	\end{align*}
\end{Definition}

\begin{Proposition}\label{propch1}
	\cite[Theorem 1.2]{adi} Let $\phi_1 \in C^s(\R) \cap C^+_{d^s}(\Om)$ be the normalized eigenvalue of $(-\De)^s$ in $X_0$.   If $q \in (0,1)$ then there exists a unique positive  $ \underline{u} \in  X_0 \cap C^+_{\phi_1}(\Om) \cap C_0(\overline{\Om}) $ classical solution to  the following problem
	\begin{equation}\label{ch25}
	(-\De)^s u
	=u^{q-1},\; u>0   \; \text{in}\;
	\Om,\; 
	u=0 \; \text{ in } \R \setminus \Om.
	\end{equation}
\end{Proposition}
\begin{Proposition}\label{Propch3}
	 Let  $ q \in (0,1), g(x,u)=u^{q-1}$ and $0 \leq H, K \in L^{\frac{2N}{N-\mu+2s}}(\Om) + L^{\frac{2N}{N-\mu}}(\Om) $. Let   $  u \in X_0$ be a positive weak solution of problem $(P_1)$. Then $u \in L^p(\Om)$ where $p \in [2, \frac{2N^2}{(N-\mu)(N-2s)})$. 
\end{Proposition}

\begin{proof}
	Since $0 \leq H, K,$ we see that  $\underline{u}\in X_0$ is a subsolution to problem $(P_1)$. \\
	\textbf{Claim:} $\underline{u}\leq u $ a.e in $\Om$. \\
Assuming by contradiction, assume that the  Claim is not true. Since for any $ u \in X_0$ we have 
\begin{align*}
\|u^+\|^2 \leq \int_{ Q} \frac{(u(x)-u(y))(u^+(x)-u^+(y))}{|x-y|^{N+2s}}~dxdy. 
\end{align*}
	Testing $(-\De)^s\underline{u}- (-\De)^s u\leq \underline{u}^{q-1}- u^{q-1} $ with $(\underline{u}-u)^+$, we obtain 
	\begin{align*}
	0\leq \|(\underline{u}-u)^+\|^2 & \leq \int_{ Q} \frac{((\underline{u}-u)^+(x)-(\underline{u}-u)^+(y))((\underline{u}-u)(x)-(\underline{u}-u)(y))}{|x-y|^{N+2s}}~dxdy\\
	& \leq \int_{ \Om} ( \underline{u}^{q-1}- u^{q-1})(\underline{u}-u)^+~dx \leq 0. 
	\end{align*}
 It implies $|\{  x \in \Om \; : \; \underline{u } \geq u \text{ a.e in } \Om \}| = 0$. It provides the expected  contradiction. Hence $\underline{u}\leq u $ a.e in $\Om$.\\	
%Since   $ q \in (0,1)$ and $u$ be a positive solution to the following problem
%	 \begin{equation*}
%	 (-\De)^s u
%	 =u^{q-2}u + \left(\ds \int_{\Om}\frac{H(y)u(y)}{|x-y|^{\mu}}dy\right) K   \; \text{in}\;
%	 \Om,
%	 u=0 \; \text{ in } \R \setminus \Om.
%	 \end{equation*}
	 Observe that   using Proposition \ref{propch1}, for   all $\ba >0$, we have  
	 \begin{align*}
	 \chi_{\{u< \ba \} }u^{q-1}\leq  \chi_{\{u< \ba \} }\frac{u}{\underline{u}^2} u^{q} < \chi_{\{u< \ba \} }\frac{u}{C_1^2\phi_1^2} \ba^{q} \leq  \chi_{\{u< \ba \} }\frac{u}{C_1^2C_2^2d^{2s}} \ba^{q}. 
	 \end{align*}
	where  $C_1$ and $C_2$ are appropriate positive constants. Hence we  can choose $\de:= \de(\ba)>0$ such that $\chi_{\{u< \ba \} }u^{q-1}= \de(\ba) \chi_{\{u< \ba \} }\frac{u}{d^{2s}}$.  Now choose $\ba >0$ such that $\ga_1:= \frac12 - S_{H} \de(\ba)>0$ and $\ga_2:= \frac{3(r-1)}{r^2} - S_{H}\de(\ba)>0$ for $2\leq r< \frac{2N}{N-\mu}$ and with  $S_H$ defined on \eqref{ch35}.  The choice of $\ba, \de(\ba)$ and Lax-Milgram Lemma,  imply that $u$ is the unique solution of the following problem: 
	\begin{equation*}
	(-\De)^s u+ \al u  -\de(\ba) \chi_{\{u< \ba \} }\frac{u}{d^{2s}}
	= \left(\ds \int_{\Om}\frac{H(y)u(y)}{|x-y|^{\mu}}dy\right) K +\chi_{\{u\geq \ba \} }u^{q-1}  + \al u   \; \text{in}\;
	\Om,
	u=0 \; \text{ in } \R \setminus \Om
	\end{equation*}
	 where $\al>0$ is chosen  as in Proposition \ref{Propch1}. Now we will follow the same arguments as in Proposition \ref{Propch1} to achieve the result. Notice that $T= \chi_{\{u\geq \ba \} }u^{q-1}  + \al u  \in X_0^\prime$. For each $n \in \N$,  we define the bilinear form 
	 \begin{align*}
	 B_n(v, w )= & \int_{ Q} \frac{(v(x)-v(y))(w(x)-w(y))}{|x-y|^{N+2s}}~dxdy +\al \int_{ \Om} vw~dx \\&  \quad  - \int_{ \Om} \int_{\Om}\frac{H_n(y)v(y)K_n(x)w(x)}{|x-y|^{\mu}}~dxdy- \int_{ \Om}\de(\ba) \chi_{\{u< \ba \} }\frac{vw}{d^{2s}} ~dx.
	 \end{align*}
	 Using as the arguments as in Proposition \ref{Propch1}, there exist  unique  $u_n \in X_0$  such that for all $w \in X_0$ we have 
	 \begin{align*}\label{ch16}
	 B_n(u_n, w)= \int_{ \Om}T w~dx. 
	 \end{align*}
 Moreover, $u_n$ is a unique solution to  the problem
	 \begin{equation*}\label{ch17}
	 (-\De)^s u_n + \al u_n 
	 = \left(\ds \int_{\Om}\frac{H_n(y)u_n(y)}{|x-y|^{\mu}}dy\right) K_n +\de(\ba) \chi_{\{u< \ba \} }\frac{u_n}{d^{2s}}  +T \; \text{in}\;
	 \Om,	u_n=0 \; \text{ in } \R \setminus \Om.
	 \end{equation*}
	 Clearly, $u_n \rp u$ weakly in $X_0$.  Let     $u_{n, \tau}=  \max\{-\tau , \min\{  u_n,\tau \} \} $ for  $ \tau>0$ and $x \in \Om$.  Choose  $ \phi = |u_{n, \tau}|^{r-2} u_{n, \tau} \in X_0$ ($2\leq r< \frac{2N}{N-\mu}$) as the  test function in \eqref{ch12}. Using the same arguments as in Proposition \ref{Propch1}, we have 
	 \begin{equation}\label{ch18}
	 \begin{aligned}
	 \frac{3(r-1)S_s}{r^2}&  \left( \int_{ \Om} |u_{n,\tau}|^{\frac{rN}{(N-2s)}}~dx\right)^{\frac{N-2s}{N}}\\
	 & \leq  C_r \int_{ \Om}|u_{n}|^{r} + \int_{E_{n,\tau}} \int_{ \Om} \frac{|H_n(y)u_{n}(y)||K_n(x)||u_{n}(x)|^{r-1}}{|x-y|^{\mu}}dxdy \\
	 & \quad +  \int_{ \Om}\de(\ba) \chi_{\{u< \ba \} }\frac{u_n}{d^{2s}}|u_{n, \tau}|^{r-2} u_{n, \tau}~dx    + \int_{\Om}g |u_{n, \tau}|^{r-2} u_{n, \tau}~dx. 
	 \end{aligned}
	 \end{equation}
	 Consider
	 \begin{equation}\label{ch19}
	 \begin{aligned}
	 \int_{\Om}T |u_{n, \tau}|^{r-2} u_{n, \tau}~dx &  = \int_{ \Om} \chi_{\{u\geq \ba \} }(u^{q-1} + \al u) |u_{n, \tau}|^{r-2} u_{n, \tau}~dx\\
	 & \leq C(N,\mu , r, |\Om|) \left(  \int_{\Om}  |u|^{r} ~dx  + \int_{\Om}   |u_{n}|^{r} ~dx\right).
	 \end{aligned}
	 \end{equation}
	 With the help of Hardy inequality, we have   
	 \begin{equation}\label{ch20}
	 \begin{aligned}
	 \int_{ \Om} \de(\ba) \chi_{\{u< \ba \} }\frac{u_n}{d^{2s}}|u_{n, \tau}|^{r-2} u_{n, \tau}~dx  & \leq 2 \int_{E_{n,\tau}} \de(\ba)  \frac{|u_{n}|^{r}}{d^{2s}} ~dx + \int_{ \Om}   \frac{|u_{n, \tau}|^{r}}{d^{2s}}~dx \\
	 & \leq S_{H}\de(\ba) \||u_{n \tau}|^{r/2}\|^2 + 2 \int_{E_{n,\tau}} \de(\ba)  \frac{|u_{n}|^{r}}{d^{2s}} ~dx .
	 \end{aligned}
	 \end{equation}
	 Using Dominated Convergence theorem, it follows that  $\ds \lim_{\tau \ra \infty} \int_{E_{n,\tau}}   \frac{ \de(\ba)|u_{n}|^{r}}{d^{2s}} ~dx=0$. 
	 Now taking into account \eqref{ch18}, \eqref{ch19}, \eqref{ch20}, definition of $\ga_2$  and  letting $\tau \ra \infty$, we have 
	 \begin{equation*}
	 \begin{aligned}
	 \left( \int_{ \Om} |u_{n}|^{\frac{rN}{(N-2s)}}~dx\right)^{\frac{N-2s}{N}}
	 & \leq  C(N,\mu , r, |\Om|)  \left( \int_{\Om}|u_{n}|^r~dx + \int_{\Om}|u|^r~dx \right). 
	 \end{aligned}
	 \end{equation*}
	 Therefore, 
	 \begin{align*}
	 \limsup_{n \ra \infty} \left( \int_{ \Om} |u_{n}|^{\frac{rN}{(N-2s)}}~dx\right)^{\frac{N-2s}{N}} \leq  C(N,\mu , r, |\Om|)  \limsup_{n \ra \infty}  \left( \int_{\Om}|u_{n}|^r~dx + \int_{\Om}|u|^r~dx \right). 
	 \end{align*}
	 Hence, $ u \in L^r(\Om)$ for all  $ r\in \left[2, \frac{2N^2}{(N-\mu)(N-2s)}\right) $. As earlier we remark that  there exists a  positive constant $ C(N,\mu , q, |\Om|) $ such that $|u|_q\leq C(N,\mu , q,  |\Om|) |u|_{2^*_s}$.  \QED 
\end{proof}

\begin{Remark}
We highlight here that the next lemma investigates the $L^\infty(\Om)$ bound for the fractional Laplacian with critical Sobolev exponent. 
	In \cite{servadei}  authors  already  proved  this type of result for a positive solution. Here we used the ideas form \cite{ squassina,servadei} to extend the result of \cite{servadei} to any weak solution. 
\end{Remark}
\begin{Lemma}\label{lemch5}
	Let $u$ be any weak solution to the following problem 
	\begin{align}\label{ch36}
	(-\De)^su = k(x,u) \text{ in } \Om,\; u=0 \text{ in } \R \setminus \Om
	\end{align}
	where $|k(x,u)|\leq C(1+|u|^{2^*_s-1})$ and $C>0$. Then $u \in L^\infty(\Om)$. 
\end{Lemma}
\begin{proof}
	Let $u \in X_0$ be any weak solution to \eqref{ch36}. Let  $u_\tau=  \max\{-\tau , \min\{  u,\tau \} \} $ for  $ \tau>0$.   Let   $ \phi = u|u_\tau|^{r-2}  \in X_0$ ($ r\geq 2$) be a test function to problem \eqref{ch36},   then by inequality \eqref{ch1}, we deduce that
	\begin{equation}\label{ch37}
	\begin{aligned}
	|u|u_\tau|^{\frac{r}{2}-1}|^2_{2^*_s}& \leq C \|u|u_\tau|^{\frac{r}{2}-1}\|^2  \leq \frac{Cr^2}{r-1} \int_{Q} \frac{(u(x)-u(y))(\phi(x)-\phi(y))}{|x-y|^{N+2s}}~dxdy\\
	&\leq  Cr \int_{ \Om} |k(x,u)||u ||u_\tau|^{r-2}~dx\\
	& \leq Cr \int_{ \Om} |u ||u_\tau|^{r-2}+|u|^{2^*_s}|u_\tau|^{r-2}  ~dx. 
	\end{aligned}
\end{equation} 
	\textbf{Claim:} Let $r_1= 2^*_s+1$. Then $ u \in L^{\frac{2^*_s r_1}{2}}(\Om)$.\\ 
For this, consider
	\begin{equation}\label{ch38}
	\begin{aligned}
	\int_{ \Om} |u|^{2^*_s}|u_\tau|^{r_1-2}  ~dx& = \int_{ |u|\leq R}  |u|^{2^*_s}|u_\tau|^{r_1-2}  ~dx+ \int_{|u|>R} |u|^{2^*_s}|u_\tau|^{r-2}  ~dx\\
	& \int_{ |u|\leq R}  R^{2^*_s}|u_\tau|^{r_1-2}  ~dx+\left(\int_{\Om}( u^2|u_\tau|^{r-2})^{\frac{2^*_s}{2}}  ~dx\right)^{\frac{2}{2^*_s}} \left( \int_{|u|>R} |u|^{2^*_s}  ~dx\right)^{\frac{2^*_s-2}{2^*_s}}. 
	\end{aligned}
	\end{equation}
	Choose $R>0$ large enough such that 
	\begin{align}\label{ch39}
	\left( \int_{|u|>R} |u|^{2^*_s}  ~dx\right)^{\frac{2^*_s-2}{2^*_s}} \leq \frac{1}{2Cr_1}.
	\end{align}
	Taking into account \eqref{ch37}, \eqref{ch38} jointly with \eqref{ch39}, we obtain
	\begin{align*}
	|u|u_\tau|^{\frac{r_1}{2}-1}|^2_{2^*_s}& \leq Cr_1\left(\int_{ \Om} |u |^{2^*_s}~dx + \int_{ \Om}  R^{2^*_s}|u|^{2^*_s-1}  ~dx\right).
	\end{align*}
	Appealing  Fatou's Lemma as $\tau \ra \infty$,  we obtain 
		\begin{align}\label{ch41}
	||u|^{\frac{r_1}{2}}|^2_{2^*_s}& \leq Cr_1\left(\int_{ \Om} |u |^{2^*_s}~dx + \int_{ \Om}  R^{2^*_s}|u|^{2^*_s-1}  ~dx\right)<\infty.
	\end{align}
	This establishes the Claim. Now let $\tau \ra \infty$ in  \eqref{ch37}, we deduce that 
		\begin{equation*}
	\begin{aligned}
	||u|^{\frac{r}{2}}|^2_{2^*_s}& \leq Cr \int_{ \Om} |u |^{r-1}+|u|^{r+2^*_s-2}  ~dx \leq 2Cr(1+|\Om|)  \left(1+ \int_{ \Om} |u|^{r+2^*_s-2} \right).
	\end{aligned}
	\end{equation*}
	It implies that 
	\begin{equation}\label{ch40}
	\begin{aligned}
\left(1+ 	\int_{ \Om}|u|^{\frac{2^*_s r}{2}}\right)^{\frac{2}{2^*_s(r-2)}}& \leq  \mc C_r^{\frac{1}{(r-2)}}  \left(1+ \int_{ \Om} |u|^{r+2^*_s-2} \right)^{\frac{1}{(r-2)}}
	\end{aligned}
	\end{equation}
	where $\mc C_r= 4Cr(1+|\Om|)$. For $j\geq 1$, we define $r_{j+1}$ iteratively as $ r_{j+1} +2^*_s-2= \frac{2^*_s r_j}{2}$. It implies 
	\begin{align*}
	\left(r_{j+1}-2 \right)= \left(\frac{2^*_s}{2}\right)^j \left(r_1-2 \right).
		\end{align*}
		From \eqref{ch40}, we get 
		\begin{align*}
		\left(1+ 	\int_{ \Om}|u|^{\frac{2^*_s r_{j+1}}{2}}\right)^{\frac{2}{2^*_s(r_{j+1}-2)}}& \leq  \mc C_{j+1}^{\frac{1}{(r_{j+1}-2)}}  \left(1+ \int_{ \Om} |u|^{\frac{2^*_s r_j}{2}} \right)^{\frac{2}{2^*_s(r_j-2)}}
		\end{align*}
	where $\mc C_{j+1}:= 4Cr_{j+1}(1+|\Om|)$. Denote $ D_j=  \left(1+ \int_{ \Om} |u|^{\frac{2^*_s r_j}{2}} \right)^{\frac{2}{2^*_s(r_j-2)}}$, for $j\geq 1$. 
	By limiting arguments, one can easily prove that, for $j>1$, 
\begin{align*}
D_{j+1}\leq \prod_{k=2}^{j+1} \mc C_{k}^{\frac{1}{(r_{k}-2)}} D_1 \leq \mc C_0 D_1. 
\end{align*}
It implies that $|u|_\infty\leq \mc C_0 D_1$ where $D_1$ is explicitly given in \eqref{ch41}.
 \QED
\end{proof}

\section{Proof of Theorem 
\ref{thmch1} and \ref{thmch2}}
In this section we will conclude the proofs of Theorem \ref{thmch1} and Theorem \ref{thmch2}. Before this we recall the following result, which can be consulted in  \cite{RS}. 
\begin{Proposition}\label{Propch5}
	Let $\Om$ be a bounded  Lipschitz domain  satisfying the  exterior  ball condition, $g \in L^\infty(\Om)$ and $u$ be a solution of \eqref{ch24}. Then $u \in C^s(\R)$ and 
	\begin{align*}
	\|u\|_{C^s(\R)}\leq C\|g\|_{L^\infty(\Om)}
	\end{align*}
	where $C$ is a constant depending on $\Om$ and $s$. 
\end{Proposition}
\textbf{Proof of Theorem \ref{thmch1} :}  Let  $ u \in X_0$ be a positive weak solution to problem $(P)$ and  $H= F(u)/u$ and $K=f$. Then 
From  Proposition \ref{Propch1},  we get   $ u \in L^r(\Om)$ for all  $ r\in \left[2, \frac{2N^2}{(N-\mu)(N-2s)}\right) $. 
It implies $F(u) \in L^r(\Om)$ for all  $ r\in \left[\frac{2N}{2N-\mu}, \frac{2N^2}{(N-\mu)(2N-\mu)}\right)$. Observe that $
\frac{2N}{2N-\mu}< \frac{N}{N-\mu}<  \frac{2N^2}{(N-\mu)(2N-\mu)}$ and  there exists a constant $C(N,\mu, |\Om|)>0$ such that $|F(u)|_{\frac{N}{N-\mu}} \leq C(N,\mu, |\Om|) |u|_{2^*_s}$. Therefore, we infer that $\int_{\Om} \frac{F(u)}{|x-y|^\mu} ~dy \in L^{\infty}(\Om)$ and 
\begin{align*}
\bigg| \int_{\Om} \frac{F(u)}{|x-y|^\mu} ~dy\bigg|_{\infty} \leq  C(N,\mu, |\Om|) |u|_{2^*_s}. 
\end{align*}
Using the assumptions on  $f$ and $g$, we obtain 
\begin{align*}
(-\De)^s u&  = g(x,u)+ \left(\ds \int_{\Om}\frac{F(u)(y)}{|x-y|^{\mu}}dy\right) f \\
& \leq C(N,\mu, |\Om|)  (1+ |u|_{2^*_s}) (1+ |u|^{2^*_s-1})= \mc S_0 (1+ |u|^{2^*_s-1})(\text{say}). 
\end{align*}
From Lemma  \ref{lemch5}, we have $ u \in L^\infty(\Om)$. Furthermore, there exists  a  function $ C_0>0 $ independent of   $N,\mu, s$ and  $|\Om|$ such that 
\begin{align*}
|u|_{\infty} \leq  C_0  \mc S_0^{\frac{2}{(2^*-1)(2^*-2)}}D_1  
\end{align*}
\begin{align*}
\text{ with  }\quad D_1\leq\left( 1+ \left( (2^*_s+1)\mc S_0 \left(\int_{ \Om} |u |^{2^*_s}~dx + \int_{ \Om}  R^{2^*_s}|u|^{2^*_s-1}  ~dx\right) \right)^{\frac{2^*_s}{2}} 
\right)^{\frac{2}{2^*_s(2^*_s-1)}}
\end{align*}
and $R>0$ chosen large enough such that 
	\begin{align}\label{ch42}
\left( \int_{|u|>R} |u|^{2^*_s}  ~dx\right)^{\frac{2^*_s-2}{2^*_s}} \leq \frac{1}{2(2^*_s+1)\mc S_0}.
\end{align}
Now  using Proposition \ref{Propch5}, we obtain  that $u \in C^s(\R)$. \QED

%\textbf{Proof of Theorem \ref{thmch2} :} By using Proposition \ref{Propch2} and the same arguments as  in  proof of Theorem \ref{thmch1}, we have the desired result.  \QED

\textbf{Proof of Theorem \ref{thmch2} :}   From Proposition \ref{Propch3}, and the assumption on  $f$, we have 
\begin{align*}
\int_{\Om} \frac{F(u)}{|x-y|^\mu} ~dy \in L^{\infty}(\Om).
\end{align*}
Furthermore,  there exists a constant $C(N,\mu ,  |\Om|) >0$ such that $|F(u)|_{\frac{N}{N-\mu}} \leq C(N,\mu , |\Om|)  |u|_{2^*_s}$. Therefore, we infer that 
\begin{align*}
(-\De)^s u&  \leq u^{q-1}+ C(N,\mu ,   |\Om|) |u|_{2^*_s} |f|  \leq u^{q-1}+ C(N,\mu ,  |\Om|)  |u|_{2^*_s} (1+ |u|^{2^*_s-1}). 
\end{align*} 
Let $ \psi \in \mathbb{R} \ra [0,1]$ be a $C^\infty(\mathbb{R})$ convex increasing function such that $\psi^\prime(t)\leq 1$ for all $t \in [0,1]$  and $\psi^\prime(t)=1$ when $t\geq 1$. Define $\psi_\e(t)= \e \psi(\frac{t}{\e})$ then using the fact that $\psi_\e$ is smooth, we obtain $\psi_\e \ra (t-1)^+$ uniformly as $\e \ra 0$.  
%\begin{align*}
%\psi(s) =  \left\{
%\begin{array}{ll}
%0 ,  &  \text{ if }  s\leq 0 , \\
%1, & \text{ if } s\geq 1.
%\end{array} 
%\right.
%\end{align*}
%Define $\psi_\e (t) = \psi\left( \frac{t-1}{\e} \right)$ for all $t \in \mathbb{R}$ and $\e>0$. 
It implies 
\begin{align*}
(-\De)^s \psi_\e(u) \leq \psi_\e^\prime(u)(-\De)^s u&  \leq \chi_{\{ u >1\}}(-\De)^s u\\
&  \leq \chi_{\{ u >1\}} (u^{q-1}+ C(N,\mu , |\Om|) |u|_{2^*_s} (1+ |u|^{2^*_s-1}))\\
&  \leq \max\{ 1, C(N,\mu ,|\Om|) |u|_{2^*_s} \}(1+ ((u-1)^+)^{2^*_s-1})\\
&=  \mc S_1 (1+ ((u-1)^+)^{2^*_s-1}) \text{ (say)}. 
\end{align*}
Hence, as $\e \ra 0$,  we deduce that
\begin{align*}
(-\De)^s (u-1)^+  \leq  \mc S_1 (1+ ((u-1)^+)^{2^*_s-1}). 
\end{align*}
Employing Lemma \ref{lemch5}, we deduce that $(u-1)^+ \in L^\infty(\Om)$, that is, $ u \in L^\infty(\Om)$.
Furthermore,  since $u$ is a positive solution, there exists $C_1>0 $ such that 
 independent of   $N,\mu, s $ and  $|\Om|$ such that 
\begin{align*}
|u|_{\infty} \leq 1+  C_1 \mc S_1^{\frac{2}{(2^*-1)(2^*-2)}}D_1  
\end{align*}
\begin{align*}
\text{ with  }\quad D_1\leq\left( 1+ \left( (2^*_s+1)\mc S_1 \left(\int_{ \Om} |(u-1)^+ |^{2^*_s}~dx + \int_{ \Om}  R^{2^*_s}|(u-1)^+|^{2^*_s-1}  ~dx\right) \right)^{\frac{2^*_s}{2}} 
\right)^{\frac{2}{2^*_s(2^*_s-1)}}
\end{align*}
and $R>0$ chosen large enough such that 
\begin{align*}
\left( \int_{|u|>R} |(u-1)^+|^{2^*_s}  ~dx\right)^{\frac{2^*_s-2}{2^*_s}} \leq \frac{1}{2(2^*_s+1)\mc S_1}.
\end{align*}
  Let  $\overline{u}$ be  the  unique solution (See \cite[Theorem 1.2, Remark 1.5]{adi}) to the following problem
\begin{align*}
(-\De)^s \overline{u} = \overline{u}^{-q}+ c, u>0 \text{ in } \Om, u =0 \text{ in } \R\setminus \Om
\end{align*}
where $c= C_1|F(u) f(u)|_\infty $ with $C_1= \bigg|\ds \int_{\Om}\frac{dy}{|x-y|^{\mu}}\bigg|_\infty$.  Then  following similar lines  as in  the proof of Claim in Proposition \ref{Propch3}, we have  $\underline{u}\leq u \leq \overline{u}$ a.e in $\Om$ where  $\underline{u}$   is  the   unique solution to \eqref{ch25}.  Therefore, $ u \in X_0  \cap L^\infty(\Om)\cap C^+_{\phi_1}(\Om) $.  Now from \cite[Theorem 1.3]{gts1}, we have the desired result.\QED

\section{Applications}
The purpose of this  section is to derive applications from the uniform estimates given in Theorems \ref{thmch1} and \ref{thmch2}.  Precisely, here, we prove the theorem \ref{thmch3} which deals with $H^s$ versus $C^0_d(\overline{\Om})$ weighted   minimizers. Furthermore, we provide an application of  this result, concerning the existence and multiplicity of solutions. 

\textbf{Proof of Theorem \ref{thmch3}: (i) implies (ii).}   
Assume by contradiction that there exists a sequence $v_n \ra v_0$ in $C^0_d(\overline{\Om})$ and $J(v_n)< J(v_0)$.  It follows that 
\begin{align*}
\int_{\Om}G(x,v_n)~dx  \ra \int_{\Om}G(x,v_0)~dx  \text{ and } \iint_{\Om\times \Om} \frac{F(v_n)F(v_n)}{|x-y|^\mu} ~dxdy \ra \iint_{\Om\times \Om} \frac{F(v_0)F(v_0)}{|x-y|^\mu} ~dxdy.
\end{align*}
Taking into account above statements, we infer that $ \ds\limsup_{n \ra \infty} \|v_n\|^2 \leq \|v_0\|^2$. Hence upto a subsequence $v_n$ converges to $v_0$  weakly in $X_0$. By Fatou's Lemma and above conclusion one obtains $\| v_n\|\ra \|v\| $. This settles the proof. \\
\textbf{Proof of Theorem \ref{thmch3}: (ii) implies (i).}
To show the result, we will first consider the case $v_0=0 $. It implies that 
\begin{align}\label{ch29}
\inf_{v \in X_0 \cap \bar{B}_\rho ^d(0)} J(v)= J(v_0)=0.
\end{align}
Assume that (i) doesn't hold. Then we can choose $\e_n \in (0,\infty),\; \e_n \ra 0$  such that there exist $z_n \in \bar{B}_{\e_n} ^X(0)$ with
 $J(z_n)<0$. For each $m\in \N$, define the functions $g_m,\;G_m: \overline{\Om} \times \mathbb{R} \ra \mathbb{R}$ and $f_m,\; F_m: \mathbb{R} \ra \mathbb{R}^+$ as 
\begin{align*}
& g_{m}(x,t) =  \max\{g(x,-m) , \min\{  g(x,t),g(x,m)\} \}, \quad G_m(x,t):= \int_0^tg_m(x,\tau) ~d \tau\\
&  \text{ and } f_m(t):= \max\{f(-m) , \min\{  f(t),f(m)\} \},\quad F_m(t):= \int_0^tf_m(\tau) ~d \tau. 
\end{align*}
Subsequently, we define the truncated functional $J_m$ as 
\begin{align*}
J_m(v)= \frac{\|v\|^2}{2} - \int_{ \Om} G_m(x,v)~dx - \frac12 \iint_{ \Om\times \Om} \frac{F_m(v)F_m(v)}{|x-y|^\mu} ~dxdy \text{ for all } v \in X_0. 
\end{align*}
Notice that $J_m \in C^1(X_0)$ and by appealing Dominated convergence theorem, we infer that $J_m(v) \ra J(v)$ as $m \ra \infty$ and  for all $v \in X_0$. 
Thus, for every $n \in \N$ we pick $ m_n \in \N$ such that $J_{m_n}(z_n)<0$. 
Observe that $|G_m(x,v)|\leq (|g(x,-m)|+|g(x,m)|)|v|$ and $F_m(v)|\leq (f(-m)+f(m))|v|$. That is, $G_m$ and $F_m$ has subcritical growth in  the sense of Sobolev inequality and  Hardy-Littlewood-Sobolev inequality respectively. Therefore, $J_m$ is weakly lower semicontinuous functional. Since $ \bar{B}_{\e_n} ^X(0)$ is a closed convex set, it implies that 
there exists $w_n \in \bar{B}_{\e_n} ^X(0)$ such that 
\begin{align*}
J_{m_n}(w_n)= \inf_{v \in \bar{B}_{\e_n} ^X(0)} J_{m_n}(v) \leq J_{m_n}(w_n).
\end{align*}
With the help of Lagrange multiplier's rule, one can easily prove that there exists $\la_n \in (0,1]$ such that $w_n$ is a weak solution of 
\begin{equation*}
\left\{\begin{array}{rllll}
(-\De)^s u
&=\la_n \left(g_{k_n}(x,u)+ \left(\ds \int_{\Om}\frac{F_{k_n}(u)(y)}{|x-y|^{\mu}}dy\right) f_{k_n}(u) \right) \; \text{in}\;
\Om,\\
u&=0 \; \text{ in } \R \setminus \Om.
\end{array}
\right.
\end{equation*}
Since $\|w_n\|\in \bar{B}_{\e_n} ^X(0),\; \|w_n\|\ra 0 $ as $|\e_n| \ra 0 $. It implies  $ |w_n|_{2^*_s} \ra 0$ and hence for $n$ large enough we  can  choose $R=0$ in \eqref{ch42}. Subsequently there exists $K>0$ such that $|w_n|_\infty\leq K$ for all $n$. By appealing \cite[Theorem 1.2]{RS}, we obtain that  for all $n$, $ w_n \in C_d^0(\overline{\Om})$ and $\|w_n\|_{C^{0,\al}_d(\overline{\Om})}\leq K_1$ for some suitable  $K_1>0$. Since $ C^{0,\al}_d(\overline{\Om})$ is compactly embedded into $C_d^0(\overline{\Om}) $, $w_n $ is strongly convergent in $C_d^0(\overline{\Om})$. Consequently, taking in account the fact that $w_n \ra 0 $ a.e in $\Om$, we get $ w_n \ra 0 $ in  $C_d^0(\overline{\Om})$. We conclude that for $n$ large enough,   $\|w_n\|_{C^0_d(\overline{\Om})}\leq \rho$ and $|w_n|_\infty <1$. From this we infer that
\begin{align*}
J(w_m)= J_{m_n}(w_m)<0
\end{align*} 
and we obtain the desired contradiction to the assumption \eqref{ch29}.
Now we will consider the case $v\not = 0$. By given assumption (ii), it follows that $ J^\prime(v_0)(v)=0$ for all $v \in C_c^\infty(\Om)$ and applying the  standard density arguments we infer that 
\begin{align}\label{ch30}
J^\prime(v_0)(v)=0 \text{ for all } v \in X_0. 
\end{align}
In view of Theorem \ref{thmch1}, we have $ u \in L^\infty(\Om) \cap C_d^0(\overline{\Om})$. For all  $v \in X_0$, let
\begin{align*}
 \widehat{F}(x,v) := & \left(\ds \int_{\Om}\frac{F(v_0+v)(y)}{|x-y|^{\mu}}dy\right) F(v_0+v)(x)-  \left(\ds \int_{\Om}\frac{F(v_0)(y)}{|x-y|^{\mu}}dy\right)\left( F(v_0)+2f(v_0)v\right)(x)\\
 \text{and }\quad
& \widehat{G}(x,v) := G(x,(v_0+v)(x)) - G(x,v_0(x))- g(x,v_0(x))v(x).
\end{align*}
Set 
\begin{align*}
\widehat{J}(v)= \frac{\|v\|^2}{2}- \int_{\Om} \widehat{G}(x,v) ~dx -\frac12 \int_{\Om} \widehat{F}(x,v) ~dx \text{ for all } v \in X_0. 
\end{align*}
Note that $ \widehat{J}\in C^1(X_0)$. Employing \eqref{ch30} and the definition of $\widehat{F} $ and $\widehat{G}$, we have 
\begin{align*}
\widehat{J}(v)& = \frac{\|v_0+v\|^2}{2}- \frac{\|v_0\|^2}{2} - \int_{ \Om} G(x,(v_0+v)(x)) - G(x,v_0(x))~dx\\
& \quad  -\frac{1}{2} \iint_{\Om\times \Om}\frac{F(v_0+v)F(v_0+v)}{|x-y|^{\mu}}~dxdy +\frac{1}{2} \iint_{\Om \times \Om}\frac{F(v_0)F(v_0)}{|x-y|^{\mu}}~dxdy\\
& = \widehat{J}(v_0+v)- \widehat{J}(v_0). 
\end{align*}
We may deduce that $\tilde{J}(0)=0$. Therefore given assumptions can be expressed as 
\begin{align*}
\inf_{v \in X_0 \cap \bar{B}_\rho ^d(0)} \widehat{J}(v)=0.
\end{align*}
Now by using above case we get the desired result and hence the proof of (ii) implies (i). \QED

\begin{Theorem}\label{thmch4}
		Let  $ G: \overline{\Om} \times \mathbb{R} \ra \mathbb{R}$ be  a   Carath\'eodory function  satisfying
	\begin{align*}
	g(x, u) =		O(|u|^{2^*_s-1}), & \text{ if } |u|\ra \infty
	\end{align*}
	uniformly for all $ x \in \overline{\Om}$. Let $f$ satisfies $(\mc F)$. Let $f(\cdot)$ and  $G(x,\cdot)$ be non decreasing functions for all $x \in \Om$. Suppose $\underline{w}, \overline{w} \in X_0$ are a weak subsolution and a weak supersolution, respectively to $(P)$, which are not solutions. Then, there exists  a solution $ w_0\in X_0$ to $(P)$ such that $\underline{w} \leq w_0 \leq \overline{w}$ a.e in $\Om$ and $w_0$ is a local minimizer of $J$ on $X_0$. 
\end{Theorem}
\begin{proof}
	Consider a closed convex  set $W$ of $X_0$ as 
\begin{align*}
W: = \{ w \in X_0\; :\;  \underline{w} \leq w_0 \leq \overline{w} \text{ a.e in } \Om   \}.
\end{align*}
Using the definition of $W$, one can easily prove that 
\begin{align*}
J(w)\geq \frac{\|w\|^2}{2}- c_1-c_2
\end{align*}
for  appropriate positive constants $c_1$  and $c_2$. This implies $J$ is coercive on $W$. $J$ is weakly lower semi continuous on $W$. Indeed, let $\{ v_n\}  \subset W$ such that $v_n \rp v$ weakly in $X_0$ as $n\ra \infty$. For each $n$, 
\begin{align*}
& \int_{ \Om} G(x,v_n)~dx \leq \int_{ \Om} G(x,v)~dx< +\infty,\\
&  \iint_{\Om \times \Om}\frac{F(v_n)F(v_n)}{|x-y|^{\mu}}~dxdy\leq  \iint_{\Om \times \Om}\frac{F(\overline{w})F(\overline{w})}{|x-y|^{\mu}}~dxdy< +\infty.
\end{align*}
Now we may invoke Dominated convergence theorem  and the weak lower semicontinuity of norms to get that $J$ is weakly lower semi continuous on $W$. Hence, there exists $ w_0 \in X_0$  such that 
\begin{align}\label{ch31}
\inf_{w \in W} J(w) = J(w_0).
\end{align}
\textbf{Claim:} $w_0$ is a weak solution to $(P)$. \\
Let $\phi \in C_c^\infty(\Om)$ and $\e>0$. Define 
\begin{align*}
u_\e = \min\{ \overline{w}, \max\{  \underline{w}, w_0+ \e\phi\} \}= v_0+\e\phi - \phi^{\e}+\phi_{\e}
\end{align*}
where $\phi^{\e}= \max\{  0, w_0+ \e\phi -\overline{w} \}$ and $\phi_{\e}= \max\{  0,  \underline{w}-w_0- \e\phi\}$. Observe that $\phi_{\e},\phi^{\e} \in X_0 \cap L^\infty(\Om)$. In view of the fact that $w_0+ t(u_\e-w_0) \in W$ for all $ t \in (0,1)$ and \eqref{ch31}, we obtain
\begin{align*}
\int_{ \R}(-\De)^s w_0 (u_\e-w_0)~dx - \int_{\Om} g(x,w_0)(u_\e-w_0)~dx- \iint_{\Om \times \Om}\frac{F(w_0)f(w_0) (u_\e-w_0)}{|x-y|^{\mu}}~dxdy \geq0. 
\end{align*}
Set 
\begin{align*}
& A^\e= \int_{ \R}(-\De)^s (w_0-\overline{w}) \phi^{\e} ~dx + \int_{ \R}(-\De)^s \overline{w}\phi^{\e} ~dx- \int_{\Om} g(x,w_0)\phi^{\e}~dx\\
 &\qquad - \iint_{\Om \times \Om}\frac{F(w_0)f(w_0) \phi^{\e}}{|x-y|^{\mu}}~dxdy, \\
& A_\e= \int_{ \R}(-\De)^s (w_0-\underline{w}) \phi_{\e} ~dx + \int_{ \R}(-\De)^s \underline{w}\phi_{\e} ~dx- \int_{\Om} g(x,w_0)\phi_{\e}~dx\\
&\qquad - \iint_{\Om \times \Om}\frac{F(w_0)f(w_0) \phi_{\e}}{|x-y|^{\mu}}~dxdy .
\end{align*}
Then by simple computations, we get 
\begin{align}\label{ch32}
\int_{ \R}(-\De)^s w_0 \phi~dx - \int_{\Om} g(x,w_0)\phi~dx- \iint_{\Om \times \Om}\frac{F(w_0)f(w_0) \phi}{|x-y|^{\mu}}~dxdy \geq \frac{1}{\e} \left( A^\e- A_\e \right). 
\end{align}
Using the assertions as in  \cite[Propostion 3.2]{gts1} with $\overline{w}$ in spite of $u_{\la^\prime}$,  we have 
\begin{align*}
\frac{1}{\e}\int_{ \R}(-\De)^s (w_0-\overline{w}) \phi^{\e} ~dx\geq o(1) \text{ as } \e \ra 0^+.
\end{align*}
To this end, employing the fact that $\overline{w} $, we deduce that 
\begin{align*}
& \frac{1}{\e} \int_{ \R}(-\De)^s \overline{w}\phi^{\e} ~dx-  \frac{1}{\e}\int_{\Om} g(x,w_0)\phi^{\e}~dx-  \frac{1}{\e}\iint_{\Om \times \Om}\frac{F(w_0)f(w_0) \phi^{\e}}{|x-y|^{\mu}}~dxdy\\
&\geq   \frac{1}{\e} \int_{\Om}(g(x,\overline{w})- g(x,w_0))\phi^{\e}~dx+ \frac{1}{\e} \iint_{\Om \times \Om}\frac{(F(\overline{w})f(\overline{w})-F(w_0)f(w_0)) \phi^{\e}}{|x-y|^{\mu}}~dxdy\\
& \geq   \frac{1}{\e} \int_{\Om}(g(x,\overline{w})- g(x,w_0))\phi^{\e}~dx = o(1) \text{ as } \e \ra 0^+.
\end{align*}
Hence we infer that $ \frac{1}{\e} A^{\e}\geq o(1)\text{ as } \e \ra 0^+$. On the similar lines, one can prove that $ \frac{1}{\e} A_{\e}\leq o(1)\text{ as } \e \ra 0^+$.  From \eqref{ch32},  for all  $\phi \in C_c^\infty(\Om)$,  it follows that 
\begin{align}
\int_{ \R}(-\De)^s w_0 \phi~dx - \int_{\Om} g(x,w_0)\phi~dx- \iint_{\Om \times \Om}\frac{F(w_0)f(w_0) \phi}{|x-y|^{\mu}}~dxdy \geq 0 \text{ as } \e \ra 0^+.
\end{align}
As $\phi \in C_c^\infty(\Om)$ was arbitrarily chosen, it implies that $w_0$ is weak solution to $(P)$. From this, we follows that there exists  a solution $ w_0\in X_0$ to $(P)$ such that $\underline{w} \leq w_0 \leq \overline{w}$ a.e in $\Om$. To prove that $w_0$ is a local minimizer in $X_0$, we proceed as follows.  Using Theorem \ref{thmch1} and  \cite[Theorem 1.2]{RS}, we deduce $ w_0 \in C^{0,\al}_d(\overline{\Om})$. Now consider
\begin{align*}
(-\De)^s(w_0-\underline{w})& \geq( g(x,w_0)- g(x,\underline{w})) +  \left( \int_{\Om }\frac{F(w_0)}{|x-y|^{\mu}}~dy\right) f(w_0)- \left( \int_{\Om }\frac{F(\underline{w})}{|x-y|^{\mu}}~dy\right) f(\underline{w})\\
& \geq 0 .
\end{align*}
 Using the fact that $\underline{w}$ is not solution to $(P)$, we have  $w_0 \not = \underline{w}$ and  by definition, $w_0-\underline{w}\geq 0$ in $\R \setminus \Om$. From \cite[Lemma 2.7]{squassina}, we infer that $w_0 -\underline{w}>Cd^s$ for some $C>0$. On a similar note $\overline{w}-w_0 > Cd^s$ for some $C>0$.  For each $w \in \bar{B}^d_{C/2}(w_0)$, we have  
\begin{align*}
\frac{w_0 -\underline{w}}{d^s} = \frac{w_0 -w}{d^s}+ \frac{w -\underline{w}}{d^s} \geq  \frac{C}{2}. 
\end{align*}
From above, it can read that $w_0 -\underline{w}>0 $ in $\Om$. Likewise, $\overline{w}-w_0>0$ in $\Om$. Therefore,  $w_0$ emerge as a local minimizer of $J$ on $X_0\cap \bar{B}^d_{C/2}(w_0)$ and this completes the proof. \QED
\end{proof}
\begin{Remark}
Consider the following problem 
\begin{equation}\label{ch33}
\left\{\begin{array}{rllll}
(-\De)^s u
&=\la \left( |u|^{q-2}u + \left(\ds \int_{\Om}\frac{F(u)(y)}{|x-y|^{\mu}}dy\right) f(u)\right) ,\; u>0 \; \text{in}\;
\Om,\\
u&=0 \; \text{ in } \R \setminus \Om,
\end{array}
\right.
\end{equation}
where $\la>0,\; 1<q<2$ and $f$  is a non decreasing function and  satisfies $(\mc F)$. Let $\underline{v}$ denotes the solution to 
\begin{equation*}
(-\De)^s u
=\la |u|^{q-2}u , \; u>0  \; \text{in}\;
\Om,
u=0 \; \text{ in } \R \setminus \Om,
\end{equation*}
and let $\overline{v}$ is a solution to 
\begin{equation*}
(-\De)^s u
=1 , \; u>0  \; \text{in}\;
\Om,
u=0 \; \text{ in } \R \setminus \Om.
\end{equation*}
Then  for all $\la>0,\; \underline{v}$ is a subsolution to \eqref{ch33}. And for $\la$ small enough, $\overline{v}$ is a supersolution to \eqref{ch33}. Now using Theorem \ref{thmch4}, there exists a solution to \eqref{ch33}, which is a local minimizer in $X_0$. The moutain pass lemma provides then the existence of a second solution.
\end{Remark}

\end{document}